\documentclass[12pt]{article}

\usepackage{amsmath,amsthm,amssymb,graphicx}

\def\N{\mathbb N}
\def\N0{\mathbb N_0}

\def\real{\mathbb R}

\def\graph{{\mathcal G}}

\def\edgeset{{\mathcal E}}
\def\tree{{\mathcal T}}
\def\treebar{\overline{\tree}}
\def\graphbar{\overline{\graph}}
\def\edgelength{l}

\def\dop{{\mathcal L}}
\def\domain{{\mathcal D}}

\def\alg{{\mathcal A}}

\title{\bf Dirichlet to Neumann Maps \\
for Infinite Quantum Graphs}
\author{Robert Carlson \\
Department of Mathematics \\ 
University of Colorado at Colorado Springs \\
rcarlson@uccs.edu}

\newtheorem{thm}{Theorem}[section]

\newtheorem{lem}[thm]{Lemma}
\newtheorem{prop}[thm]{Proposition}

\theoremstyle{definition}

\theoremstyle{remark}

\newcommand{\thmref}[1]{Theorem~\ref{#1}}

\newcommand{\lemref}[1]{Lemma~\ref{#1}}
\newcommand{\propref}[1]{Proposition~\ref{#1}}

 \numberwithin{equation}{section}

\begin{document}

\maketitle

\date{}

\begin{abstract}

The Dirichlet problem and Dirichlet to Neumann map are 
analyzed for elliptic equations on   
a large collection of infinite quantum graphs.
For a dense set of continuous functions on the graph boundary,
the Dirichlet to Neumann map has values in the Radon measures on the graph boundary.

\end{abstract}

\vskip 25pt

2000 Mathematics Subject Classification 34B45

\vskip 25pt

\newpage

\section{Introduction}

The recent surge of activity in analysis on graphs is wide ranging, encompassing 
the spectral theory of finite graphs \cite{Chung,ColinB},
physics inspired problems on finite or infinite quantum graphs \cite{Exner}, 
and resistance network models \cite{Colin,Jorgensen,KL2}, often related to probability \cite{Doyle,Lyons,Woess},
to mention a few highlights. The study of harmonic functions is a common theme, 
particularly when the work is related to probability.  The demands of harmonic function theory and probability
have inspired studies of infinite tree boundaries \cite{Cartier,Cohen} and some work on
boundaries for more general graphs \cite{Georg11,Jorgensen,Woess2}. 

Extending the previous work beyond harmonic functions and trees,
this paper treats elliptic boundary problems, in particular the Dirichlet to Neumann map,
for a large class of infinite quantum graphs.
Broadly speaking, Dirichlet to Neumann maps describe the relationship between
the value of a function $f:\partial B \to \real $ on the boundary $\partial B $
of some spatial domain $B $ and the normal derivative at the boundary of an extension 
$u:B \to \real $ of $f$.  Typically, the function $u$ satisfies a differential equation 
in $B $, with solutions being uniquely determined by the boundary values $f$.
Physically, this formulation is used to describe the current flowing out of
a domain in response to an applied boundary voltage, or the heat flux at the boundary
in response to a fixed boundary temperature distribution. 

Calderon \cite{Calderon} inspired substantial work on the inverse problem of determining (nonconstant)
interior electrical conductivities in a domain in $\real ^N$ from the voltage to current
map at the boundary; see \cite{Uhlmann} for an overview. 
Related Steklov eigenvalue and expansion problems are treated in \cite{Auchmuty}. 
The Dirichlet to Neumann map and related problems have also been 
studied for a variety of network models.  Classical resistor networks are considered in \cite{Curtis}.
A quantum graph Dirichlet to Neumann map was used in \cite{Brown} to develop a finite 
tree version of the Borg-Levinson inverse eigenvalue theorem.
Boundary control methods as developed in \cite{Avdonin,Belishev} provide an alternative
technology for Dirichlet to Neumann mapping problems.

The class of quantum graphs treated here is motivated by the challenge of modeling
enormously complex biological networks such as the circulatory system \cite{Carlson06}, the nervous system \cite{Nicaise}, 
or the pulmonary network \cite{Maury,Zelig}.  The graphs $\graph $ are typically    
infinite, but have finite diameter and compact metric completions $\graphbar $.   
They also satisfy a 'weakly connected' condition introduced in \cite{Carlson08}.
Not only do these conditions provide productive idealizations of biological networks, but
they also appear in the study of continuous time Markov chains with 'explosions' \cite{Carlson11}. 

This work resolves a number of distinct problems while extending the earlier work on harmonic functions on 
infinite trees to encompass elliptic boundary problems for a large class of infinite quantum graphs.  
The results include an existence theorem for the Dirichlet problem for a class of elliptic equations.
The subsequent study of the Dirichlet to Neumann map is motivated by the problem of discussing flows through the 
graph boundary.  Classical work on the Dirichlet to Neumann map is usually handled in the context of spatial
domains with smooth boundaries.  Part of the challenge here is that the boundaries $\partial \graphbar $
of infinite graph completions usually lack any differentiable structure, so alternative formulations are required.

The second section begins with a review of developments from \cite{Carlson08}.
Properties of weakly connected metric graph completions and their algebra $\alg$ of test functions are 
recalled.  These graph completions have totally disconnected boundaries, so there is a rich collection
of clopen sets, that is sets which are both open and closed, which play an important role in the analysis.
The elliptic equations $u'' = qu$ with $q \ge 0$ and standard quantum graph vertex conditions 
are then introduced. A maximum principle is established for these equations; it proves to be a key tool
in the analysis.  A solution of the Dirichlet problem for harmonic functions
developed in \cite{Carlson08} is extended to include the more general elliptic equations and a larger 
class of graphs.  

The third section introduces the Dirichlet to Neumann map, beginning with finite graphs,
where classical derivatives at the boundary are available. 
Computations for the $\alpha - \beta $ tree illustrate the extraordinarily complex behavior
of derivatives of simple harmonic functions on infinite graphs.  Using the voltage to current map physical interpretation
for guidance, it proves productive to think in terms of current flow through a boundary set rather than
current flow at a point.  This idea is first implemented using test functions from $\alg $ to define 
a Dirichlet to Neumann function $\Lambda _q(F,\Omega )$.  Here $F$ may be  
any continuous function on the boundary $\partial \graphbar $ of the graph completion, but $\Omega $ 
is restricted to be a clopen subsets of $\partial \graphbar $.     

To ensure more regularity of the Dirichlet to Neumann function with respect to the set $\Omega $,
the fourth section considers the case when $F$ is the characteristic function $1_{\Omega (1)}$ 
of a clopen subset of $\partial \graphbar $.
The restricted functions $\Lambda _q(F,\Omega )$ extend in a standard way to signed Borel measures on $\partial \graphbar$
with the expected positivity when $F = 1_{\Omega (1)}$ and $\Omega \subset \Omega (1)$. 
In the final section, an extension by linearity in the first argument yields a  
Dirichlet to Neumann map which is a densely defined nonnnegative symmetric operator from the continuous functions
on $\partial \graphbar $ to the dual space of Radon measures.  

\section{Foundations}

This section reviews some of the terminology and results from \cite{Carlson08}, which should be
consulted for proofs and additional information.  
After introducing the geodesic metric, weakly connected metric graph completions are described.
These graph completions have totally disconnected boundaries, with a rich collection of
clopen sets, which are both open and closed.  
Weakly connected graph completions have a closely associated algebra $\alg$ of 'eventually flat' functions.  
Following this review, the elliptic equations $u'' = qu$ with $q \ge 0$ and standard quantum graph vertex conditions 
are introduced, and a maximum principle established.  A solution of the Dirichlet problem for harmonic functions
developed in \cite{Carlson08} is extended to include the more general elliptic equations and a larger 
class of graphs.    

\subsection{Metric graphs}

$\graph $ will denote a locally finite graph with a countable vertex set.
Interior vertices are those with more than one incident edge, 
while boundary vertices have a single incident edge.  
$\graph $ is assumed to be connected, but connectivity is not assumed
for subgraphs that appear in the course of proofs.

To study differential equations on $\graph $, edges $e_n$ in the edge set $\edgeset$ 
are identified with real intervals of finite length $\edgelength _n$. 
Loops and multiple edges with the same vertices are not directly considered, but can be accomodated by adding vertices. 
In the usual quantum graph style, the Lebesgue measure for intervals may be extended to 
a measure for $\graph $.  The Euclidean length on the intervals is extended to paths of finitely
many intervals by addition, and then to a 'geodesic' distance between points $p_1$ and $p_2$ by   
taking the infimum of the lengths of paths joining $p_1$ and $p_2$.   

As a metric space, ${\cal G}$ has a completion $\graphbar$ which is assumed to be compact in this work.
Without much difficulty \cite[Prop. 2.1]{Carlson08} one sees that    
$\graphbar$ is compact if and only if for every $\epsilon > 0$ there
is a finite set of edges $S = \{ e_1,\dots , e_n \}$ such that for every $y \in \graph $
there is a edge $e_k \in S$ and a point $x_k \in e_k$ such that $d(x_k,y) < \epsilon $.
Define the boundary $\partial \graph$ of $\graph$ to be the set of boundary vertices of $\graph $.
The interior of $\graph $ or $\graphbar $ is $\graph \setminus \partial \graph $.
The boundary $\partial \graphbar $ of the completion $\graphbar$ is 
defined to be the union $\partial \graphbar = \partial \graph \cup [\graphbar \setminus \graph ] $. 
Since $\partial \graphbar$ is closed, it too is compact. 
 
Graphs and their completions are more amenable to analysis when they satisfy an additional condition.
Say that $\graphbar $ is {\it weakly connected} if
for every pair of distinct points $x,y \in \graphbar $,
there is a finite set of points $W = \{ w_1,\dots ,w_K \} \subset \graph $ 
separating $x$ from $y$.  That is, there are disjoint open sets $U,V$ with
$\graphbar \setminus W = U \cup V$, and with $x \in U$, $y \in V$.
An alternative characterization is that there is a finite set of edges $W_{\edgeset}$ from $\graph $
such that every path from $x$ to $y$ contains an edge from $W_{\edgeset}$.

Completions of trees are weakly connected.  Less obviously, so are the completions of graphs with finite volume 
\cite{Carlson08}.  The following result, similar to one in \cite{Carlson11}, shows that trees can be modified
substantially without losing a weakly connected completion.

\begin{thm}
Suppose the graph $\graph $ is obtained from a tree $\tree $
by adding a sequence of edges $e_n$ whose lengths $l_n$ satisfy 
$\lim_n l_n = 0$.  Assume there is a positive constant $C$ such that 
\[d_{\graph}(x,y) \le d_{\tree}(x,y) \le C d_{\graph}(x,y), \quad x,y \in \tree.\]
Then $\graphbar $ is weakly connected.
\end{thm}

\begin{proof}

We haven't changed the set of Cauchy sequences of vertices, so 
$\graphbar \setminus \graph = \treebar \setminus \tree $.
Suppose $x$ and $y$ are distinct points in $\graphbar \setminus \graph $.
Let $E$ be an edge with length $L$ on the path from $x$ to $y$ in $\tree $.
Let $U_x$ be the connected component of $\treebar \setminus E$ containing $x$,
and let $V_y$ be $\treebar \setminus (U_x \cup E)$.
Find $N$ so that $l_n < L/C$ for $n > N$, and for $n \le N$ remove edges $e_n$,  along
with $E$, from $\graphbar $, leaving $\graphbar _1$.  
Suppose there is a path in $\graphbar _1$ connecting $U_x$ to $V_y$.
Then there must be vertices $u$ and $v$ separated by $E$ in $\tree $ but joined by an edge $e_n$
whose length, by assumption, is smaller than $L/C$.  Then we have
$d_T(u,v) \ge L$, but $d_G(u,v) < L/C$, a contradiction.
Thus cutting $E$ together with $e_1,\dots ,e_N$ provides a separation of 
$x$ and $y$.

\end{proof}

If $\graphbar $ is weakly connected, then $\partial \graphbar$ is totally disconnected, meaning that
connected components of $\partial \graphbar $ are points.         
Totally disconnected compact metric spaces have a rich collection of clopen sets, 
which are both closed and open.  In fact \cite[p. 97]{Hocking}, if $\Omega $ is a totally disconnected compact metric space,
then for any $x \in \Omega $ and $\epsilon > 0$, there is a clopen set $U$
which contains $x$ and is contained in the $\epsilon $ ball centered at $x$.

The connectivity properties of $\graphbar $ are closely related to 
the algebra $\alg$ of 'eventually flat' functions 
$\phi : \graphbar \to \real $ 
which are continuous, infinitely differentiable on the open edges of $\graph $, 
and for which $\phi '(x) = 0$ for all $x$ in the complement of a finite set of edges, and
in a neighborhood of each vertex $v \in \graph $.
With pointwise addition and multiplication, $\alg$ is an algebra.
Since the constant functions are in $\alg$, 
the following result \cite[Lemma 3.5]{Carlson08} may be combined with the Stone-Weierstrass
Theorem to show that if $\graphbar $ is weakly connected and compact, then $\alg $ is uniformly 
dense in the real continuous functions on $\graphbar $.

\begin{lem} \label{testfunks}
Suppose $\graphbar $ is weakly connected.  If 
$\Omega $ and $\Omega _1$ are disjoint compact subsets of 
$ \graphbar $,
then there is a function $\phi \in \alg$ such that $0 \le \phi \le 1$,
\[ \phi (x) = 1, \quad x \in \Omega , \quad  
 \phi(y) = 0, \quad y \in \Omega _1. \]
\end{lem}

\subsection{Forms and operators}

Suppose $q:\graph \to \real$ is nonnegative and locally integrable.
The main focus of this work is the study of certain solutions of the equation 
\[u'' = qu, \quad u:\graph \to \real .\]
The equation is to hold in the sense that $u'$ is absolutely continuous, and the equation holds a.e.
on the edges of $\graph $.  In addition  
the function $u$ is assumed to satisfy the standard continuity and derivative junction conditions
at vertices $v$ with degree at least $2$ in $\graph $, 
\begin{equation} \label{intvc}
\lim_{x \in e(i) \to v} u_i(x) = \lim_{x \in e(j) \to v}u_j(x), 
\quad e(i),e(j) \sim v.
\end{equation}
\[\sum_{e(i) \sim v} \partial _{\nu} u_i(v) = 0.\]
Here the derivative $\partial _{\nu} u_i(v) = u'$ is computed in local
coordinates identifying $[a,b]$ with the edge $e(i)$ of length $b-a$,
and $a$ is the coordinate value for $v$.
Functions $u:\graph \to \real$ satisfying the equation and vertex conditions will be called  
$q$-harmonic functions. 

Given an edge weight function $\omega : \edgeset \to (0,\infty )$ 
constant on each edge, define the weighted inner product 
\[\langle f,g\rangle _{\omega } = \sum_e \int_e f(x)g(x) \omega  (e) \ dx ,\]
and corresponding Hilbert space $L^2_{\omega  }$.  Let 
$\domain _{\alg ,\omega  } = \alg \cap L^2_{\omega  }$.   
Introduce the form
\begin{equation} \label{biform}
Q(f,g) = \int_{\graph } f'(x)g'(x) + q(x)f(x)g(x),
\end{equation}
with domain $\domain _{\alg ,\omega  } $,   
and the differential expression
\[L_{\omega  }f =  \frac{1}{\omega  (e)} [-f''(x) + q(x)f(x)].\]   
For $f,g \in \domain _{\alg ,\omega }$, 
integration by parts leads to the alternate form
\[Q(f,g) = \sum_{e \in \edgeset } \int_e \frac{1}{\omega  (e)} [-f''(x) + q(x)f(x)]g(x) \omega  (e)
= \sum_{e \in \edgeset } \int_e [L_{\omega }f(x)]g(x) \omega  (e).\]
These computations are summarized in the following proposition.
 
\begin{prop}
The bilinear form
\[Q(f,g), \quad f,g \in \domain _{\alg ,\omega  },\]
is densely defined in $L^2_{\omega }$, has a nonnegative quadratic form $Q(f,f)$,
and satisfies 
\[Q(f,g) = \langle L_{\omega  }f,g \rangle _{\omega } = \langle f, L_{\omega  }g \rangle _{\omega } .\]
\end{prop}

The form $Q$ has an associated inner product \cite[pp. 308-318]{Kato} 
\[\langle f,g \rangle _{\omega  ,1} = \langle f,g \rangle _{\omega  } + Q(f,g),\]
with the corresponding norm $\| f \| _{\omega  ,1}$.
Completing the domain $\domain _{\alg ,\omega  }$ with respect to this inner product, 
construct the Hilbert space $H^1(\omega  )$. 

Integration from $x$ to $y$ over a path $\gamma $ of length at most $4 d(x,y)$ yields the estimate 
\begin{equation} \label{simpint}
|f(y) - f(x)|^2 = |\int_{\gamma } f'(t) \ dt |^2
\le \int_{\gamma } 1 \ dt \int_{\gamma } |f'(t)|^2 \ dt \le 4d(x,y) Q(f,f).
\end{equation}
Since $\graph $ is connected with finite diameter, the arguments of \cite{Carlson08} show
that functions $f$ in $H^1(\omega  )$ are continuous on $\graph $.  There is a constant $C$ such that 
\[\sup_{x \in \graph }|f(x)| \le C\| f \| _{\omega  ,1},\]
and $f$ satisies a Lipshitz  estimate
\begin{equation} \label{Lipest}
|f(y) - f(x)| \le 2d(x,y)^{1/2}\| f \| _{\omega  ,1}.
\end{equation}  
Consequently, $f$ has a unique continuous extension to $\graphbar$. 

Let $\domain _0$ denote the set of functions $f \in \alg $
whose support is contained in the union of a finite collection of edges, and  
such that $f(x) = 0 $ for all $x$ in a neighborhood of any boundary vertex.
Let $S_0$ denote the symmetric nonnegative operator acting by $L_{\omega }$ 
on the domain $\domain _0$ in $L^2_{\omega }$.
$S_0$ has a self adjoint Friedrich's extension $\dop _0$ acting on $L^2_{\omega }$.
If $q$ is not bounded, the description of the domain of $S_0$ is more delicate than described.
The reader should either consult \cite[pp. 343--346]{Kato} for the necessary modifications,
or simply assume that $q$ is bounded. 

Suppose $\omega  (\graph ) = \sum_e \omega  (e) < \infty $.  
If $f \in \domain _0$ with $\| f \| _{\omega  } = 1$, then 
there is some point $y \in \graph $ with $|f(y)|^2 \ge \omega  ^{-1}(G)$.  If the connected graph $\graph $ 
has a boundary vertex or has infinitely many edges, then there will be some $x \in \graph $ with $f(x) = 0$.
Then \eqref{simpint} gives
\[\omega  ^{-1}(\graph ) \le |f(y)|^2 = |f(y) - f(x)|^2 \le 4R Q(f,f),\]
leading to the next result.

\begin{prop} \label{lowbnd}
Suppose $\omega  (\graph ) = \sum_e \omega  (e) < \infty $ and 
$\partial \graphbar \not= \emptyset $.  Then the self adjoint operator $\dop _{0}$
on $L^2_{\omega }$ has a strictly positive lower bound.  That is, there is a 
$ C > 0$ such that for all $f$ in the domain of $\dop _0$ with $\| f \| _{\omega } = 1$,
$\langle \dop _0 f,f \rangle _{\omega } \ge C$.
\end{prop}

\subsection{$q$-harmonic functions}

There is a  maximum principle for $q$-harmonic functions.

\begin{lem} \label{maxprin}
Suppose $q$ is a nonnegative locally integrable function, and
$u:\graph \to \real $ is $q$ - harmonic.
If $u$ has a positive global maximum at an interior point of $\graph $,
then $u$ is constant.
\end{lem}

\begin{proof}
Arguing by contradiction, suppose $u$ is nonconstant with a positive global maximum
at the interior point $x_2 \in \graph $.  Find another interior point $x_0$ with $u(x_0) < u(x_2)$.  
The first step is to find an interior point $x_1$ with $u(x_1) = u(x_2)$ such that
every neighborhood of $x_1$ contains points $x$ with $u(x) < u(x_1)$.
Recalling that $\graph $ is connected,
let $\gamma :[0,1] \to \graph $ be a path in the interior of $\graph $
with $\gamma (0) = x_0$ and $\gamma (1) = x_2$.  If
\[t_0 = \sup \ \{ t \in [0,1] \ | \ u(\gamma (s)) < u(x_2) \ { \rm for \ all } \ s < t \} .\]
then $x_1 = \gamma (t_0)$ has the desired properties.

If $x_1$ is not a vertex, then $u'(x_1) = 0$ by Calculus.
If $x_1$ is an interior vertex, then $u_n'(x_1) = 0$ 
for each edge $e_n$ incident on $x_1$ by the vertex conditions \eqref{intvc}.

The function $u''$ is locally integrable, so \cite[p. 110]{Royden} for $x$ near $x_1$, 
\[u'(x) = \int_{x_1}^x u''(t) \ dt = \int_{x_1}^x q(t)u(t) \ dt \ge 0.\]
Since $u$ is not constant in any neighborhood of $x_1$,
$u'(x) > 0$ for some $x$ arbitrarily close to $x_1$.  
Since $u'$ is continuous on each edge of $\graph $, 
\[u(x) - u(x_1) = \int_{x_1}^x u'(t) \ dt > 0, \]
so $x_1$ is not a local maximum.
\end{proof}

\begin{lem} \label{max01}
With the hypotheses of \lemref{maxprin},
suppose $u:\graphbar \to \real $ is continuous, nonconstant, $q$-harmonic on $\graph $,
and $0 \le u(x) \le 1$ for all $x \in \partial \graphbar$.
Then $0 < u(x) < 1$ for all $x \in \graphbar \setminus \partial \graphbar $.

\end{lem}

\begin{proof}
By \lemref{maxprin} $u(x) < 1$ in $\graphbar \setminus \partial \graphbar $.  
Similarly, if $u(x_0) < 0$ for some $x_0 \in \graphbar$, then $-u$,
has a positive interior maximum, which is not possible.

This leaves the case when $u(x_0) = 0$ for some $x_0 \in \graphbar \setminus \partial \graphbar $.  
Interior points of edges may be treated as vertices of degree two, so our focus is on 
a vertex $v$ with degree at least two, with $u(v) = 0$.
Let $u_i$ denote the restriction of $u$ to an edge $e_i$ incident on $v$. 
Since $u_i(x) \ge 0$, the derivatives, computed in 
outward pointing local coordinates, satisfy $\partial _{\nu} u_i(v) \ge 0$.
The vertex conditions \eqref{intvc} then force $u_i'(v) = 0$ for each incident edge $e_i$.
On $e_i$ the function $u_i$ satisfies the ordinary differential equation $u_i'' = qu_i$, 
with initial condition $u(v) = u'(v) = 0$, so $u(x) = 0$ for all $x \in e_i$.
These observations also imply that $u$ is identically $0$ along any path once it vanishes
at any point in $\graphbar \setminus \partial \graphbar $.  

\end{proof}

Extending Theorem 4.2 of \cite{Carlson08} with a similar proof, the next result shows that if 
$\omega  ^{-1}q \in L^2(\omega  )$, then the Dirichlet problem is solvable for
$q$-harmonic functions.  In particular the Dirichlet problem for the usual harmonic functions 
($q = 0$) is solvable whenever $\graphbar $ is weakly connected and compact. 

\begin{thm} \label{Dprob}
Suppose $\graph $is connected, while $\graphbar $ is weakly connected and compact. 
Assume there is a weight $\omega  $ such that $\omega  (\graph ) = \sum_e \omega  (e) < \infty $, and
the nonnegative locally integrable function $q$ satisfies
\begin{equation} \label{qbnd}
\int_{\graph } (\omega  ^{-1}q)^2 \omega  < \infty .
\end{equation}
If $\partial \graphbar \not= \emptyset$,
then every continuous function $U:\partial \graphbar \to \real $ has a unique
extension to a continuous function $u:\graphbar \to \real $
that is $q$-harmonic on $\graph $.
\end{thm}

\begin{proof}
As noted in the remarks preceeding \lemref{testfunks}, $\alg $ is uniformly 
dense in the real continuous functions on $\graphbar $.
After extending $U$ to a continuous function on $\graphbar $, find a function $g_n \in \alg$ with
\[\max_{x \in\graphbar} |U(x) - g_n(x)| \le 1/n.\]
Since $g_n \in \alg $, the function $D^2g_n$ is bounded with support contained in the union of
a finite collection of edges.  Since $g_n$ is bounded, \eqref{qbnd} implies that
\[ \omega  ^{-1}[D^2 - q]g_n \in L^2(\omega  ).\]
The operator $\dop _0$ is invertible, so the function 
\begin{equation} \label{hdef}
h_n = \dop _0^{-1} \omega  ^{-1}[D^2 - q]g_n 
\end{equation}  
satisfies
\[\omega  ^{-1}[-D^2 + q]h_n = \omega  ^{-1}[D^2 - q]g_n. \]  
Since $g_n \in \alg$ and $h_n$ is in the domain of $\dop _0$,
$u_n = h_n + g_n$ satisfies the vertex conditions \eqref{intvc}.
In addition $h_n(x) = 0$ for $x \in \partial \graphbar $.
That is, $u_n$ is $q$-harmonic on $\graph $ with 
$u_n(x) = g_n(x)$ for $x \in \partial \graphbar $.

Suppose $m < n$.  Then 
\[\max_{x \in \partial \graphbar } |u_m - u_n| \le \frac{2}{m}.\]
Since both $\pm (u_m - u_n)$ are $q$-harmonic on $\graph $, 
the maximum principle implies 
\[\max_{x \in \graphbar } |u_m - u_n| \le \frac{2}{m},\]
so $\{ u_n \} $ is a uniformly Cauchy sequence on $\graphbar $,
with a continuous limit $u$.

On each edge $e$ the solutions of $u'' = qu$ form a two dimensional vector space.
Recall \cite[pp. 4-6]{Kato} that any two norms on a finite dimensional vector
space induce the same topology, so 
the uniform convergence to $0$ of $|u_m(x) - u_n(x)|$ for $x \in e$ implies   
the uniform convergence to $0$ of $|u_m(x) - u_n(x)| + |u_m'(x) - u_n'(x)| $.
Recasting solutions of $u'' = qu$ as solutions of an integral equation
and taking limits shows that $u(x) = \lim_n u_n(x)$ is the desired $q$-harmonic function.   

\end{proof}

\section{Introducing the Dirchlet-Neumann map}

Beginning with finite quantum graphs, where classical derivatives at the boundary are available, 
this section introduces the Dirichlet to Neumann map. 
Sample computations for a family of infinite $\alpha - \beta $ trees demonstrate 
the complex behavior possible for derivatives of harmonic functions on infinite graphs.
For continuous boundary functions $U$ and clopen sets $\Omega \subset \partial \graphbar$,  
a preliminary Dirichlet to Neumann function $\Lambda _q(U,\Omega )$ is defined using test functions
from $\alg $.  Since the existence of $q$-harmonic extensions is fundamental, the hypotheses of \thmref{Dprob}
are assumed.  Moreover, $q$ is subsequently assumed to be integrable over $\graph $.

\subsection{Finite Graphs }

In case $\graph $ is a finite connected metric graph,    
$\partial \graph $ is just the set of vertices of degree one.
The hypotheses of \thmref{Dprob} are satisfied if $\omega  (e) = 1$ for each edge $e$, so
for any function $U:\partial \graph \to \real $ there is a unique $q$ - harmonic
extension $u:\graph \to \real $.  The function $u$ will have an absolutely continuous derivative.
For each boundary vertex $v$, let $\partial _{\nu}$ denote the 
derivative computed with respect to local coordinates pointing outward at $v$.
The Dirichlet to Neumann map $\Lambda _q$ acting on the vector space of
functions on $\partial \graph $ is defined by
\[\Lambda _q U(v) = \partial _{\nu} u(v), \quad v \in \partial \graph .\]

For a finite graph, introduce the Hilbert space $l^2(\partial \graph )$
of functions $F: \partial \graph \to \real$, 
with the usual inner product 
\[ \langle F,G \rangle _{\partial \graph} 
= \sum_{v \in \partial \graph} F(v) G(v).\]
Extending functions $F: \partial \graphbar \to \real $ to smooth functions $f:\graph \to \real $ 
satisfying the standard interior vertex conditions \eqref{intvc}, integration by parts gives
\begin{equation} \label{IBP1}
 \int_{\graph} (-f'' + qf) g = 
 - \sum_{v \in \partial \graph} (\partial _{\nu}f(v)) g(v)
+ \int_{\graph} f'g' + qfg.
\end{equation}
With $Lf = -f'' + qf$, a second integration gives
\begin{equation} \label{IBP2}
\int_{\graph} (Lf) g
= \sum_{v \in \partial \graph} [f(v)\partial _{\nu}g(v) - g(v)\partial _{\nu}f(v)]
+ \int_{\graph} f (Lg).
\end{equation}

If $Lf = 0$ and $g = 1$, \eqref{IBP2} reduces to 
\begin{equation} \label{currcon}
\sum_{v \in \partial \graph} \partial _{\nu}f(v) = \int_{\graph} f (Lg).= \int_{\graph} f q.
\end{equation}
When $q = 0$ this equation corresponds to 'conservation of current'.  
In case $Lf = Lg = 0$, \eqref{IBP2} reduces to 
\[  \langle \Lambda _q F,G \rangle _{\partial \graph}  
= \sum_{v \in \partial \graph} g(v)\partial _{\nu} f(v) 
= \sum_{v \in \partial \graph} f(v)\partial _{\nu}g(v) 
=\langle F,\Lambda _q G \rangle _{\partial \graph}  . \]
If, in addition, $G = F$, then \eqref{IBP1} gives
\begin{equation} \label{DNform}
\langle \Lambda _q F,F \rangle _{\partial \graph}  = \int_{\graph} (f')^2 + qf^2.
\end{equation}
Finally, if $\Lambda _qF = 0$, then \eqref{DNform} with
$q \ge 0$ shows that $f' = 0$ everywhere,
which implies $f$ is constant on $\graph $, and $q = 0$ almost everywhere.  
These basic computations are summarized in the following proposition.

\begin{prop} \label{fingcase}
For a finite graph $\graph $ with nonempty boundary $\partial \graph $, 
the Dirichlet to Neumann map $\Lambda _q$ is self adjoint and nonnegative on $l^2(\partial \graph )$.
The null space of $\Lambda _q$ is $\{0\}$ unless $q = 0$ almost everywhere, in which case the null space 
consists of the constants. 
\end{prop}

Using the $q$-harmonic extensions $f$ of $F$ and $g$ of $G$, \eqref{IBP1} expresses the bilinear form 
$\langle \Lambda _qF,G\rangle _{\partial \graph } $ as an integral over $\graph $,
\[ \langle \Lambda _qF,G\rangle _{\partial \graph } =  \int_{\graph} f'g' + qfg.\]
It is also possible to express $\Lambda _qF$ using test functions in \eqref{IBP2}.
Let $\phi (x)$ be a smooth function with $\phi _e'(w) = 0$ for
all edges $e$ and vertices $w \in \graph $.  If $\phi (v) = 1$ at the boundary
vertex $v$, while $\phi (w) = 0$ at all other boundary vertices, then \eqref{IBP2} gives
\begin{equation} \label{Fweakform}
\Lambda _qF(v) = \partial _{\nu}f(v) =  \int_{\graph} f L \phi .
\end{equation}

\subsection{Infinite network example: the $\alpha -\beta $ tree}

Since the derivatives of harmonic functions on infinite graphs can exhibit complex behavior,
extending the Dirichlet to Neumann map from finite to infinite networks will require new ideas.
To illustrate this complexity, consider an infinite tree $\tree _{\alpha }$, as illustrated in Figure \ref{fig:abtree}.
$\tree _{\alpha }$ has a root vertex $v_0$ of degree $1$ with adjacent vertex $v_1$. All
vertices except $v_0$ have degree $3$.  Organize the vertices into combinatorial depth levels, so $v_0$
is at level $0$, $v_1$ is at level $1$, and if $v_n$ is at level $n$, its children are at level $n+1$.

\begin{figure}[htp]
\begin{center}
\includegraphics[width=3in]{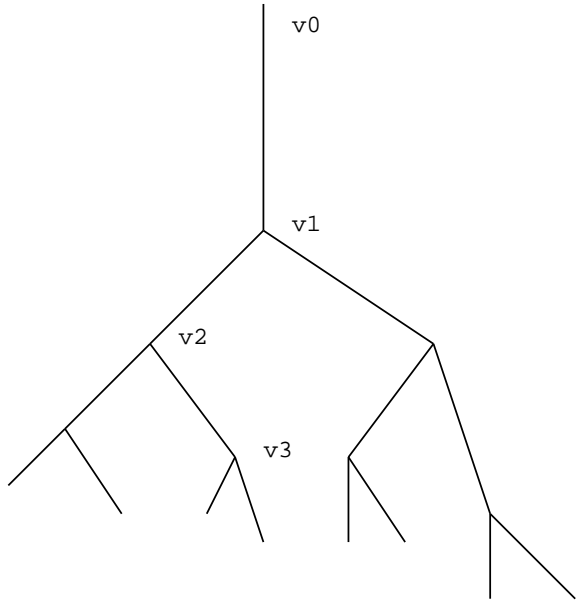}
\caption{An $\alpha - \beta$ tree} 
\label{fig:abtree}
\end{center}
\end{figure}

The edge weights are described by a scaling parameter $\alpha $, where 
$\alpha $ and $\beta = 1-\alpha $ satisfy $0 < \alpha , \beta < 1 $.
For any vertex $v \in \tree $ other than $v_0$, label the single incident parent edge $e_p(v)$, 
and the two child edges $e_{\alpha }(v)$ and $e_{\beta }(v)$, with lengths $l_p$, $l_{\alpha }$ and $l_{\beta} $ respectively.  
The lengths satisfy $l_{\alpha } = \alpha l_p$, $l_{\beta } = \beta l_p$.

Given a linear function $u(x) = cx + d$ on the edge $[v_0,v_1]$, extend $u$ to $\tree _{\alpha }$
by requiring $u$ to be continuous on $\tree _{\alpha }$ and linear on each edge, with
\[u'(x) = \frac{-\beta }{\alpha + \beta }u'(t), \quad t \in e_p(v) 
, \quad x \in e_{\alpha}(v),\] 
\[u'(x) = \frac{-\alpha }{\alpha + \beta }u'(t), \quad t \in e_p(v) 
, \quad x \in e_{\beta}(v).\] 
The derivatives are computed in local coordinates pointing outward from $v$..
One checks easily that $u$ is harmonic on $\tree _\alpha $, 
has the same value at all vertices on the same level of $\tree $,  
and has a finite limit as the level index $n \to \infty$.

Paths descending through the levels are determined by sequences $\{ \sigma _n \}$ with
$\sigma _n \in \{ \alpha , \beta \}$.  The derivative below the vertex $v_N$ at level $N$
will be   
\[u'(v_N^+) = u'(v_1^-)\prod_{n=1}^N \frac{\alpha \beta }{\sigma _n}, \quad u'(v_1^-) = c.\]   
Note that $\lim_{N \to \infty} u'(v_N^+) = 0$ independent of the path, but the decay behaviors
are path dependent and diverse.     

\subsection{Test function formulation}

As noted in \eqref{Fweakform}, there is a test function formulation
of the Dirichlet-Neumann map for finite graphs.  For infinite graphs
the points in $\partial \graphbar$ are generally not isolated, so
the previous formulation is not applicable.
An alternative is available when $\graphbar $ is compact and weakly connected,
with functions in $\alg$ serving as test functions and clopen 
subsets of $\partial \graphbar $ generalizing boundary vertices.  

The set $\partial \graphbar$ is a closed subset of $\graphbar$, 
and $\partial \graphbar$ is totally disconnected if $\graph $ is weakly connected.
Since $\partial \graphbar$ is a totally disconnected compact metric space,
every open ball $B_{\epsilon }(x)$ contains a clopen neighborhood of $x$.
Applying \lemref{testfunks}, the next lemma guarantees a plentiful supply
suitable test functions.  The succeeding lemma describes their features. 

\begin{lem} \label{clopenfunk}
Suppose $\graphbar $ is compact and weakly connected.  If $\Omega $ is
a clopen subset of $\partial \graphbar$, then there is a function $\phi \in \alg $
such that $0 \le \phi \le 1$ and
\[\phi (x) = 1, \quad x \in \Omega , \quad \phi (y) = 0, 
\quad y \in \partial \graphbar \setminus \Omega .\] 
\end{lem}

\begin{lem} \label{bndeq}
If $\phi \in \alg$, then the restriction $\phi :\partial \graphbar \to \real$
has finite range $R$. The set $\{ x | \phi (x) \notin R \}$ is contained in the union
of finitely many edges.
\end{lem}

\begin{proof}
Let $E$ be a nonempty finite set of closed edges containing all points $x$ with $\phi ' (x) \not= 0$.
For any point $z \in \partial \graphbar \setminus E$ there is a path 
$\gamma :[0,1] \to \graphbar$ with $\gamma (0) \in E$,  $\gamma (1) = z$, and $\gamma (t) \in \graph $  
for $0 < t < 1$ by \cite[Prop 2.3]{Carlson08}. Since $d(z,E) > 0$ there is a smallest number $t_0$
with $0 \le t_0 < 1$, and such that $\gamma (t) \notin E$ for all $t > t_0$.
Since $E$ is the union of a finite set of closed edges, $\gamma (t_0)$ is an endpoint
of one of these edges, and $\phi (z) = \phi (\gamma (t_0))$, since $\phi '(\gamma (t)) = 0$
for all $t > t_0$. This shows that       
the restriction $\phi :\partial \graphbar \to \real$
has finite range $R$, and that $\{ x | \phi (x) \notin R \} \subset E$.
\end{proof}

Assume the hypotheses of \thmref{Dprob} are satisfied.
Before treating the Dirichlet to Neumann map, a preliminary step 
is to consider the real valued Dirichlet-Neumann function $\Lambda _q(U,\Omega)$,
defined for continuous functions $U:\partial \graphbar \to \real$
and clopen sets $\Omega \subset \partial \graphbar $.
The value of $\Lambda _q(U,\Omega)$ plays the role of the integral of the normal derivative of the $q$-harmonic 
extension $u$ over set $\Omega \subset \partial \graphbar $.
Using \eqref{Fweakform} as a guide, suppose $\phi \in \alg$ satisfies 
\[\phi (x) = 1, \quad x \in \Omega , \quad \phi (x) = 0, 
\quad x \in \partial \graphbar \setminus \Omega .\] 
The Dirichlet-Neumann function $\Lambda _q(U,\Omega)$ is defined by the formula
\begin{equation} \label{DNmeas}
\Lambda _q(U, \Omega ) = \int_{\graph } u L \phi ,
\end{equation}
The next lemma shows that the dependence on $\phi $ is illusory.

\begin{lem} \label{wellpose}
Suppose $\Omega \subset \partial \graphbar $ is clopen, and that $\phi ,\psi \in \alg $ with
\[\phi (x) = \psi (x) = 1, \quad x \in \Omega , \quad \phi (x) = \psi (x) = 0, 
\quad x \in \partial \graphbar \setminus \Omega .\] 
Then
\[\int_{\graph } u L \phi  = \int_{\graph } u L \psi .\]
\end{lem}

\begin{proof}
The integration by parts formula \eqref{IBP2} will be used for certain finite
subgraphs $\graph _1$ of $\graph $. The validity of this formula is based on 
the assumption that the interior vertex conditions \eqref{intvc} hold at vertices
of degree greater than one.  A subgraph $\graph _1$ may include vertices $v$ 
with degree $d_1(v)$ in $\graph _1$, but with strictly larger degree $d(v)$ in $\graph $.
To resolve this problem, the boundary $\partial \graph _1$ is understood in a relative
sense, comprising the set of vertices $v \in \graph _1$ with $d_1(v) = 1$
or with $d_1(v) < d(v)$.  In the formula \eqref{IBP2},
such a vertex $v$ is treated as $d_1(v)$ boundary vertices of degree one.

Using \lemref{bndeq} and the fact that $\phi , \psi \in \alg $, 
there is a finite connected subgraph $\graph _0 \subset \graph $
such that for all connected finite subgraphs $\graph _1$ with 
$\graph _0 \subset \graph _1 \subset \graph $ we have
\[0 = \phi '(x) = \psi '(x) = \phi (x) - \psi (x), \quad x \in \graph \setminus \graph_1,\]
and 
\[0 = \phi '(x) = \psi '(x) = \phi (x) - \psi (x), \quad x \in \partial \graph_1.\]
Since $Lu = 0$,the formula \eqref{IBP2} gives
\[\int_{\graph _1} u L\phi  = \sum_{v \in \partial \graph _1} \phi(v)\partial _{\nu}u(v) 
= \sum_{v \in \partial \graph _1} \psi(v)\partial _{\nu}u(v) 
= \int_{\graph _1} u L\psi .\]  

This equality gives  
\[|\int_{\graph } u L \phi  - u L \psi | = |\int_{\graph \setminus \graph _1} u L \phi  - u L \psi | 
= |\int_{\graph \setminus \graph _1} u q (\phi  - \psi ) |.\]
Since $u(\phi - \psi )$ is bounded and $q$ is integrable, the right hand side can be made
arbitrarily small by enlarging $\graph _1$. 
\end{proof}

Having established a Dirichlet-Neumann function $\Lambda _q(U,\Omega)$, a couple of
preliminary observations are in order.  The first concerns nontriviality. 

\begin{prop}
If $U: \partial \graphbar \to \real $ is nonnegative, continuous, and somewhere positive,
and if $q(x)$ is integrable and strictly positive, then $\Lambda _q(U,\partial \graphbar ) > 0$. 
\end{prop}

\begin{proof}
By \lemref{max01} the $q$ harmonic extension $u$ of $U$ is everywhere nonnegative.  
For the clopen set $\Omega = \partial \graphbar $ simply 
take $\phi $ to be the constant function $1$.
Then 
\[\Lambda _q(U,\partial \graphbar ) = \int_{\graph} uq > 0 .\]
\end{proof}

For the second observation, suppose the hypotheses of \lemref{wellpose} hold.
Assume that $\{ \graph _n \}$  is an exhaustion of $\graph $ by a sequence of finite subgraphs, so that  
\[\graph _n \subset \graph _{n+1}, \quad \graph = \bigcup _n \graph _n.\]
For all $n$ sufficiently large  
\[\phi '(x) = 0, \quad \phi (x) \in \{ 0,1 \} , \quad x \in \partial \graph_n.\]
Using the relative notion of graph boundary, let $B_n = \phi ^{-1}(1) \cap \partial \graph _n $. 
Applying \eqref{IBP2} as in the proof of \lemref{wellpose} leads to an approximation 
of $\Lambda _q(U,\Omega )$ using finite graph Dirichlet to Neumann maps,
\begin{equation} \label{finiteapp}
\Lambda _q(U,\Omega ) = \int_{\graph } u L\phi  = 
\lim_{n \to \infty} \sum_{v \in B_n} \partial _{\nu}u(v) .
\end{equation}

\section{Characteristic functions}

This section focuses on the Dirichlet to Neumann function $\Lambda _q(U,\Omega _1 )$
when $U = 1_{\Omega }:\partial \graphbar \to \real$ is the characteristic function of a nonempty clopen set,
\[1_{\Omega }(x) = \Bigr \{ \begin{matrix} 
1, & x \in \Omega , \cr
0, & x \in \partial \graphbar \setminus \Omega .
\end{matrix} \Bigl \} \]
Let $u_{\Omega }$ denote the $q$-harmonic extension of $1_{\Omega}$ to $\graphbar$. 
By \lemref{max01}, if $\Omega \not= \partial \graphbar$, then
the inequality $0 < u_{\Omega }(t) < 1$ holds for all $t \in (\graphbar \setminus \partial \graphbar ) $.

For guidance, consider a finite graph $\graph $ and point $x \in \Omega $ with $u(x) = 1$.
The Mean Value Theorem together with the equation $u_{\Omega }'' = qu_{\Omega} \ge 0$
imply that $u_{\Omega}'(x) > 0$.  For $y \in (\partial \graph \setminus \Omega )$ a similar argument gives $u'(y) < 0$.
This section will establish analogous results for infinite graph completions.  

\subsection{Regular values}

To provide technical support for an upcoming argument, the notion of regular values of a function is 
adapted to functions defined on networks.
Assume that $h:\graph \to \real $ is continuous, with 
a continuous derivative on the open edges of $\graph $.
A point $x \in \graph $ is a critical point for $h$ if $x$ is a vertex
or if $h'(x) = 0$.  A number $y \in \real $ is a critical value for $h$
if $h^{-1}(y)$ contains a critical point.
Points in the range of $h$ that are not critical values are regular values.
The next result is a simple form of Sard's Theorem.

\begin{lem} \label{Sard}
The set of critical values of $h$ has Lebesgue measure $0$.
\end{lem}

\begin{proof}
The set of vertices is countable, so the set of values $h(v)$ for
vertices $v$ has measure $0$.  There are countably
many edges, so it suffices to prove the result for a single edge $e$.
Suppose the length of $e$ is $l_e$, and $e$ is identified with the 
real interval $[0,l_e]$.

Since $h'$ is continuous, for $\epsilon > 0$ the set 
$U_{\epsilon} = \{ x| -\epsilon < h'(x) < \epsilon \}$ is open. 
$U_{\epsilon}$ may be written as the union of countably many pairwise disjoint open intervals 
$(\alpha _n,\beta _n) \subset (0,l_e)$. With $z_n = (\alpha _n + \beta )/2$ and    
$\alpha _n < x < \beta _n$,  
\[h(x) - h(z_n) = \int_{z_n}^x h'(t) \ dt .\]
Thus the image of the $n$-th interval has length at most $(\beta _n - \alpha _n)\epsilon $, implying    
$h(U_\epsilon )$ is contained in a set of measure at most $\epsilon l_e$. 
\end{proof}

\begin{lem} \label{finset}
Suppose $\graphbar $ is compact, and that $c$ is a regular value of $h$.
Assume there is no $x \in \partial \graphbar $ with $h(x) = c$.  
Then $h^{-1}(c)$ is a finite set containing no vertices.
\end{lem}

\begin{proof}
If $h^{-1}(c)$ were an infinite set, then by compactness there would be an infinite
sequence $\{ x_n \} \subset h^{-1}(c)$ converging to a point $z$, with $h(z) = h(x_n) = c$.
Moreover $z$ is an interior point of some edge $e$.  
By the Mean Value Theorem and the continuity of $h'$, $h'(z) = 0$,
contradicting the assumption that $c$ is a regular value for $h$.
\end{proof}

Level sets of $h$ can be used to construct subgraphs of the metric space $\graph$
with added vertices and subdivided edges.
Let $c_1,c_2$ be regular values of $h$ with $c_1 < c_2$, and assume that
$h(x) \notin [c_1,c_2]$ for all $x \in \partial \graphbar $.
Let $C =  h^{-1}(c_1) \cup h^{-1}(c_2)$, which is finite by \lemref{finset}.
Add vertices of degree two at the points of $C$,
and subdivide the corresponding edges to obtain a modified graphical
description of $\graph $. 
Let $\graph (c_1,c_2)$, the level set graph, be the set of points $x \in \graph $
such that every path from $x$ to $\partial \graphbar $ contains a point in $C$.

\begin{lem} \label{levgraph}
If $\graphbar $ is compact, then $\graph (c_1,c_2)$ is a finite graph whose  
boundary consists of points in $C$. If $v$ is a boundary vertex of $\graph (c_1,c_2)$ with $h(v) = c_2$,
then $\partial _{\nu} h(v) > 0 $, while if $h(v) = c_1$,
then $\partial _{\nu} h(v) < 0 $.  
\end{lem}

\begin{proof}
Since $\partial \graphbar $ is compact,  
there is an $\epsilon > 0$ such that $h(x) \notin [c_1,c_2]$ if $d(x,\partial \graphbar )< \epsilon$. 
This implies \cite[Prop 2.1]{Carlson08} that $\graph (c_1,c_2)$ is contained in the union 
of a finite collection $E$ of edges of the (original) graph $\graph $.
Assume that $E$ is chosen so that every edge $e \in E$ contains some point of $\graph (c_1,c_2)$.

Consider the following observations about an edge $e \in E$, identified with an interval $[a,b]$.
First, if $x_0,x_1,x_2 \in e$ with $x_0 \le x_1 \le x_2$, and $x_0,x_2 \in C$, then $x_1 \in \graph (c_1,c_2)$. 
Second, if $x_0,x_2 \in e$ with $x_0 < x_2$, and there is no point $x_1 \in C$ with $x_0 \le x_1 \le x_2$, 
then by the Intermediate Value Theorem either 
$x_0$ and $x_2$ are both in $\graph (c_1,c_2)$, or neither is. 
Third, if $x_1 \in C$, then $h'(x_1) \not= 0$ implies that an interval containing $x_1$
is a subset of $\graph (c_1,c_2)$. 

Let $N$ be the number of points of $C$ in $e$.
By considering the three cases $N=0$, $N=1$, and $N \ge 2$, one finds that 
the set $e \cap \graph (c_1,c_2)$ is a finite union of closed edges from the modified graph $\graph $
(including points of $C$ as vertices).
Consequently, $\graph (c_1,c_2)$ is a finite graph whose boundary vertices are in $C$.
The conclusions about the derivatives at boundary vertices $v$ follow from $h'(v) \not= 0$.
\end{proof}

\subsection{Boundary measures}

\begin{thm} \label{positivity}
Assume that $\graph $ is connected, and that $\graphbar $ is weakly connected and compact.
Suppose $\Omega ,E_1, E_2$ are nonempty clopen subsets of 
$ \partial \graphbar $ with $\Omega \not= \partial \graphbar $,
$E_1 \subset \Omega $, and $E_2 \subset \partial \graphbar \setminus \Omega $. 
Then
\begin{equation} \label{posneg}
\Lambda _q(1_{\Omega },E_1) > 0 ,\quad \Lambda _q(1_{\Omega },E_2) \le 0. 
\end{equation}
If $q(x)=0$ for all $x$ in some neighborhood of $\partial \graphbar $, then
$ \Lambda _q(1_{\Omega },E_2) < 0$. 
\end{thm}

\begin{proof}
The two cases in \eqref{posneg} are similar; the $E_1$ case will be emphasized. 

By \lemref{max01}, if $u_{\Omega }$ is the nonconstant $q$-harmonic extension of $1_{\Omega}$ to $\graphbar$, then
\[0 < u_{\Omega}(x) < 1, \quad x \in \graphbar \setminus \partial \graphbar .\]
Since $\graphbar $ is connected the range of $u_{\Omega }$ is $[0,1]$.   
Suppose  $\phi \in \alg $ with
\[\phi (x) = 1, \quad x \in E_1, \quad \phi (x) = 0, 
\quad x \in \partial \graphbar \setminus E_1 .\] 

The idea of the proof is to use a finite graph $\graph _0$, constructed as in \lemref{levgraph}, 
to approximate $\graph $.
The approximating graph should include all points where $\phi '(x) \not= 0$ so that
\begin{equation} \label{splitcalc}
\Lambda _q(1_{\Omega },E_1) = \int _{\graph } u_{\Omega} L \phi 
 = \int _{\graph _0} u_{\Omega}  L \phi + \int _{\graph \setminus \graph _0} u_{\Omega}  q\phi .
\end{equation}
The integral over $\graph _0$ is expressed more transparently using \eqref{IBP2}.
The possible positivity of the last integral term accounts for the slight difference
in inequalitites for $E_1$ and $E_2$ unless $q$ vanishes near $\partial \graphbar $.

By \lemref{bndeq} there is an $\delta _1 > 0$ such that if 
$d(x, \partial \graphbar \setminus E_1) < \delta _1 $ then
$\phi (x) = 0$, while if $d(x,E_1 ) < \delta _1 $ then $\phi (x) = 1$.
The integrability of $q$ implies that for any $\epsilon > 0$ there is a $\delta > 0$
such that 
\[\int _{N(\delta  )} u_{\Omega}  q \phi < \epsilon ,\quad 
N(\delta  ) = \{ x \in \graphbar \ | \ d(x,\partial \graphbar ) < \delta \} . \] 
Choose $\delta  < \delta _1$.

Pick a finite subgraph $\graph _{\delta }$ of $\graph $ such that $\graph \setminus \graph _{\delta }$
is a subset of $N(\delta )$.  The subgraph $\graph _{\delta }$ may contain boundary vertices of $\graph $, 
but by adding a finite set of new vertices to $\graph $, and subdividing the associated edges,
$\graph _{\delta }$ can be chosen so that it contains no points of $\partial \graphbar $.
By \lemref{max01} the function $u_{\Omega }$ restricted to $\graph _{\delta }$ has a minimum $\alpha _1 $
and a maximum $\beta _1$ satisfying $0 < \alpha _1 < \beta _1 < 1$.  Now pick regular values $\alpha $ and 
$\beta $ for $u_{\Omega }$ so that $0 < \alpha _1 < \alpha $ and $\beta _1 < \beta < 1$.  
Using the previous level set graph selection, define $\graph _0 = \graph (\alpha ,\beta )$.
Every boundary vertex of $\graph _0$ is in the set $N(\delta )$.

There will be a nonempty set of boundary vertices $v$ of $\graph _0$ with $u_{\Omega}(v) = \beta$ and $\phi (v) = 1$. 
By \lemref{levgraph} we then have $\partial _{\nu} u_{\Omega }(v) > 0$.
Using \eqref{IBP2} as in \eqref{finiteapp}, 
\[\int_{\graph _0} u_{\Omega } L \phi 
= \sum_{v \in \partial \graph _0 \cap \phi ^{-1}(1)} \partial _{\nu} u_{\Omega }(v) > 0. \]
Together with  
\[0 \le \int _{\graph \setminus \graph _0} u_{\Omega}  q\phi < \epsilon ,\]
the result is established.
\end{proof}

Fix a reference clopen set $\Omega $ as in \thmref{positivity}, with the corresponding
$q$-harmonic extension $u_{\Omega}:\graphbar \to \real $ of $1_\Omega : \partial \graphbar \to \real $.   
By \thmref{positivity} the function $\Lambda _q(1_{\Omega },E)$ is positive on nonempty clopen subsets 
$E$ of $\Omega $.  Consider extending the set function $\Lambda _q(1_{\Omega },E)$ to
a positive measure on $\Omega $, and a signed measure on $\partial \graphbar $.  
The construction of a measure from a premeasure \cite[p.30]{Folland} 
is used here.  Recall that an algebra of subsets of $\graphbar $ is a nonempty collection 
closed under finite unions and complements.  The collection of clopen subsets of $\partial \graphbar $
is an algebra.  

\begin{lem} \label{additive}
Assume that $E \subset \Omega $ is clopen, 
and $\{ E_n \}$ is a countable partition of $E$ by clopen sets.
Then
\[\Lambda _q(1_{\Omega },E) = \sum_n \Lambda _q(1_{\Omega },E_n)\]
\end{lem}

\begin{proof}
Since $\graphbar $ is compact and $E$ is closed, $E$ is compact.
The given partition is an open cover of $E$, so is finite.  

Finite additivity will now follow from \eqref{DNmeas}.  
For each set $E_n$ there is a function $\phi _n \in \alg $ such that, for $x \in \partial \graphbar$,
$\phi _n(x) = 1$ if $x \in E_n$ and $\phi _n(x) = 0$ if $x \notin E_n$.
The function $\phi = \sum_n \phi _n$ is in $\alg $, restricts to the characteristic
function of $E$ on $\partial \graphbar$, and 
\[\Lambda _q(1_{\Omega} ,E) = \int_{\graph } u_{\Omega} \dop \phi =
\int_{\graph } u_{\Omega} \dop \sum_n \phi _n 
= \sum_n \Lambda _q(1_{\Omega},E_n). \]
\end{proof}

\begin{thm} \label{mexist}
$\Lambda _q(1_\Omega , E)$ extends to a finite positive Borel measure on $\Omega$
and to a signed Borel measure on $\partial \graphbar $.
\end{thm}

\begin{proof}
Suppose $K \subset \Omega$ is a closed set.
Cover $K$ with open sets $U_x(N)$ contained in balls of radius less than $1/N$
centered at points $x \in K$. Since $\partial \graphbar $ is totally disconnected 
we may assume this is a cover by clopen sets.
Since $K$ is compact, a finite subcover $K \subset \{ U_m(N) \}$ exists,
and
\[K = \cap _{N = 1}^{\infty} [\cup _m U_m(N)]. \]
It follows that the sigma algebra generated by the clopen sets is the Borel sets.

\lemref{additive} shows that $\Lambda _q(1_{\Omega },E)$ is a premeasure on the algebra of clopen sets.
This premeasure has \cite[p.30]{Folland} a unique extension to a Borel measure on $\Omega $.

For $E \subset \Omega ^c$ this argument shows that $-\Lambda _q(1_{\Omega },E)$
provides a positive Borel measure on $\Omega ^c$.   
Pasting these two positive measures together yields two Borel measures on $\partial \graphbar$,
the positive Borel measure 
\[ |\Lambda _q(1_{\Omega },E)| = \Lambda _q(1_\Omega ,E\cap \Omega ) - 
\Lambda _q(1_\Omega ,E\cap \Omega ^c),\]  
and the signed Borel measure 
\begin{equation} \label{smeas}
\Lambda _q(1_{\Omega },E) = \Lambda _q(1_\Omega ,E\cap \Omega ) +
\Lambda _q(1_\Omega ,E\cap \Omega ^c).
\end{equation}

\end{proof}

The measures $\mu (E) = \Lambda _q(1_{\Omega},E)$ constructed above are finite Radon measures
\cite[Thm 7.8 and p. 216]{Folland}.    
In particular for every Borel set $E \subset \partial \graphbar $,
\[ \mu (E) = \inf \{ \mu (O), \ E \subset O, \ O \ {\rm open} \},\]
and 
\[ \mu (E) = \sup \{ \mu (K), \ K \subset E, \ K \ {\rm compact} \}.\]
With the additional assumption that $q(x)=0$ for all $x$ in some neighborhood of $\partial \graphbar $, 
\thmref{positivity} shows that the positive measure $ \sigma (E) = |\Lambda _q(1_{\Omega },E )|$ is strictly positive 
on all nonempty clopen subsets $E$ of $\Omega $.   

The following argument shows that the measures $\mu $ will be absolutely continuous with respect to the measures $\sigma $.
The first step is the next lemma, whose proof is straightforward. 

\begin{lem} \label{helper}
Suppose $K \subset \partial \graphbar $ is compact, $O \subset \partial \graphbar $ is open, and
$K \subset O$.  Then there is a finite collection of clopen sets $E_1,\dots ,E_K$ such that
$K \subset \cup E_k \subset O$. 
\end{lem}

\begin{prop}
The measures $\mu = \Lambda _q(1_{\Omega })$ are absolutely continuous with respect
to any Radon measure $\sigma $ which assigns positive measure to each clopen set.
\end{prop}

\begin{proof}
Using the regularity of our measures, for any Borel set $E$ with $\mu (E) > 0$
there is a compact set $K_{\mu}$ and an open set $O_{\mu}$ such that $K_{\mu} \subset E \subset O_{\mu}$
and $\mu (E)/2 \le \mu (K_{\mu}) \le \mu(O_{\mu})  \le 2\mu (E)$, and
similarly for $\sigma (E)$.
By taking the intersection $O_{\mu} \cap O_{\sigma}$ and the union 
$K_{\mu} \cup K_{\sigma}$, 
the same compact and open sets may be used for both measures.
By \lemref{helper} there is a finite covering of $K$ by clopen subsets $E_1,\dots ,E_K$ of $O$.
Since $\sigma $ assigns positive measure to all nonempty clopen sets, it follows that
$\sigma (O_{\sigma}) > 0$, and so $\sigma (E) > 0$.

\end{proof}

\section{Operator Theory}

Let $C(\partial \graphbar )$ denote the Banach space of continuous functions $F:\partial \graphbar \to \real $
with $\| F \| = \max_{x \in \partial \graphbar}|F(x)|$.  Let $M( \partial \graphbar )$ denote
the Banach space of signed finite Radon measures $\mu$ on $\partial \graphbar $, with 
\[ \| \mu \| = |\mu |(\partial \graphbar  ) = \sup_{\| F \| = 1} | \int_{\partial \graphbar } F \ d \mu | .\] 
As noted above, the measures $\mu (E) = \Lambda _q(1_{\Omega},E)$ are finite Radon measures.  
By the Reisz Representation Theorem \cite[pp. 216-217]{Folland} or \cite[p. 82]{Lax},      
$M(\partial \graphbar )$ is the dual of $C(\partial \graphbar )$.

The Dirichlet to Neumann function $\Lambda _q(F,E)$ maybe extended to the Dirichlet to Neumann map 
$\Lambda _q: C(\partial \graphbar ) \to M( \partial \graphbar )$.  
Suppose $F:\partial \graphbar \to \real $ is a continuous function with finite range, or 
equivalently is a linear combination of characteristic functions of clopen sets, 
$F = \sum_{k=1}^n \alpha _k 1_{\Omega (k)} $.
With this domain, define the Dirichlet to Neumann map to be the operator 
\begin{equation} \label{DNmap}
\Lambda _q F (E) = \Lambda _q(F,E) = \sum_{k=1}^n \alpha _k \Lambda _q(1_{\Omega (k)},E)
\end{equation}
Since the definition \eqref{DNmeas} is linear in $F:\partial \graphbar \to \real$,
the measure valued map is consistently defined for clopen sets $\Omega (k)$ and $E$
in $\partial \graphbar $.  The following theorem establishes symmetry of the Dirichlet to Neumann function.

\begin{thm} \label{symmetry}
Suppose $\Omega (1)$ and $\Omega (2)$ are nonempty clopen subsets of 
$ \partial \graphbar $.
Then
\begin{equation} \label{symmeqn}
\Lambda _q(1_{\Omega (1)},\Omega (2)) 
= \Lambda _q(1_{\Omega (2)},\Omega (1)).
\end{equation}
\end{thm}

\begin{proof}

Beginning with some notation,
let $\phi , \psi \in \alg $, with $0 \le \phi \le 1$, $0 \le \psi \le 1$,
\[
\phi (x) = \Bigl \{ \begin{matrix} 1,& x \in \Omega (1), \cr 
0, & x \in \partial \graphbar \setminus \Omega (1) \end{matrix} \Bigr \}, \quad
\psi (x) = \Bigl \{ \begin{matrix} 1,& x \in \Omega (2) , \cr 
0, & x \in \partial \graphbar \setminus \Omega (2) \end{matrix} \Bigr \}.\] 
For $\delta > 0$ let
\[N(\delta  ) = \{ x \in \graphbar \ | \ d(x,\partial \graphbar ) < \delta \} . \]  
For $\delta $ sufficiently small $x \in N(\delta )$ implies
\[\phi (x) = \Bigl \{ \begin{matrix} 1, & d(x,\Omega (1)) < \delta , \cr 
0, & d(x,\partial \graphbar \setminus \Omega (1)) < \delta , 
\end{matrix} \Bigr \}, \quad 
\psi (x) = \Bigl \{ \begin{matrix} 1, & d(x,\Omega (2)) < \delta , \cr 
0, & d(x,\partial \graphbar \setminus \Omega (2)) < \delta , 
\end{matrix} \Bigr \} .\]

The claimed symmetry of the Dirichlet to Neumann function is elementary if $\graph $ is finite.  
For $j = 1,2$, let $u_{\Omega (j)}$ be the $q$-harmonic extension of $1_{\Omega (j)}$.
By \eqref{Fweakform} and the symmetry of the finite
graph Dirichlet to Neumann map,
\begin{equation} \label{fincase}
\Lambda _q(1_{\Omega (1)},\Omega (2)) = \int_{\graph} u_{\Omega (1)}L\psi  
= \sum_{v \in \Omega (2)} \partial _{\nu} u_{\Omega (1)}(v)
= \sum_{\partial \graph } u_{\Omega (2)} \partial _{\nu} u_{\Omega (1)}
\end{equation}
\[ = \sum_{\partial \graph } u_{\Omega (1)} \partial _{\nu} u_{\Omega (2)}
= \sum_{\Omega (1)} \partial _{\nu} u_{\Omega (2)}
= \Lambda _q(1_{\Omega (2)},\Omega (1)) .\]

The argument proceeds by reduction of the general case to the finite graph case.
Let $\graph _n$ be an increasing sequence of finite subgraphs exhausting $\graph $. 
Assume each $\graph _n$ includes all edges with 
a point $x$ where $\phi '(x) \not= 0$ or $\psi '(x) \not= 0$.
Also assume that points in the (relative) boundary of $\graph _n$ are contained
in $N(1/n)$, with $1/n < \delta $.
Define
\[ \Omega _1(n) = \partial \graph _n \cap \{ x \ | \ \phi (x) = 1 \}, \quad   
\Omega _2(n) = \partial \graph _n \cap \{ x \ | \ \psi (x) = 1 \}.\]   

Let $f_n$ be the $q$-harmonic function on $\graph _n$ 
with boundary values $F_n(x) = \phi (x)$ for $x \in \partial \graph _n$, and  
let $g_n$ be the $q$-harmonic function on $\graph _n$ 
with boundary values $G_n(x) = \psi (x)$.  
The summands in
\[\sum_{w \in \Omega _1(n) } \partial _{\nu} f_n(w) 
= \int_{\graph _n} f_n L \phi  ,\]
are positive terms, and letting $\phi _1 = 1 - \phi $ there is a sum of negative terms
\[\sum_{w \in \partial \graph _n \setminus \Omega _1(n)} \partial _{\nu} f_n(w) 
= \int_{\graph _n} f_n L \phi _1 .\]
Since $0 \le f_n \le 1$, these expressions combine to give
\[\sum_{w \in \partial \graph _n} |\partial _{\nu}f_n(w)|
\le \int_{\graph} |L \phi | + |L \phi _1| = C_{\phi},\] 
with $C_{\phi}$ independent of $n$.

Integration by parts gives
\[0 = \int_{\graph _n}(-f_n'' + qf_n)f_n 
= -\sum_{w \in \partial \graph _n}f_n\partial _{\nu}f_n(w)
+ \int_{\graph _n}(f_n')^2 + \int_{\graph _n}qf_n^2 ,\]
so
\[\int_{\graph _n}(f_n')^2 \le C_{\phi}.\]
The Cauchy-Schwarz inequality then gives the Lipschitz estimate
\[|f_n(y) - f_n(x)| \le \sqrt{C_{\phi}{\rm dist}(x,y)} ,\]
the distance computed in $\graph _n$.  Similar estimates hold for $g_n$.

Since the sequence $\{ f_n \}$ is bounded and equicontinuous,
it has a subsequence, still denoted $\{ f_n \}$, converging uniformly to 
$f$ on finite subgraphs of $\graph $.
The argument in the last paragraph of the proof of \thmref{Dprob} shows $f$ is $q$-harmonic.
The function $f$ satisfies the same type of Lipschitz estimate, the distance now computed in $\graph $.  
Since $f$ is uniformly continuous on $\graph $, it extends continuously to $\graphbar $.
In particular,
$f(x) = 1$ for $x \in \Omega (1) $, and $f(x) = 0$ for 
$x \in \partial \graphbar \setminus \Omega (1)$, so $f = f_{\Omega (1)}$.  
Similar comments apply to $\{ g_n \}$, which has a limit function $g_{\Omega (2)}$. 

Making use of \eqref{fincase}, 
\begin{equation} \label{finapp}
\Lambda _q(1_{\Omega (1)},\Omega (2)) = \int_{\graph} f_{\Omega (1)} L \psi 
= \lim_n \int_{\graph _n} f_n L \psi  = \lim_n \Lambda _q(1_{\Omega _1(n)}, \Omega _2(n))
\end{equation}
\[ = \lim_n \Lambda _q(1_{\Omega _2(n)}, \Omega _1(n))
=  \lim_n \int_{\graph _n} g_n L \phi = \Lambda _q(1_{\Omega (2)},\Omega (1)).\] 

\end{proof}

\begin{thm}
The Dirichlet to Neumann map 
$\Lambda _q: C(\partial \graphbar ) \to M( \partial \graphbar )$
is densely defined, symmetric, and nonnegative.  
\end{thm}

\begin{proof}

As noted in the discussion preceding \lemref{clopenfunk},
distinct points $x,y \in \partial \graphbar $ are contained in disjoint
clopen sets. The Stone-Weierstrass Theorem then shows that  
the linear combinations of characteristic functions of clopen
sets are dense in $C(\partial \graphbar )$. 

Let $\langle \mu , F \rangle = \int_{\partial \graphbar } F \ d \mu $ denote the dual pairing of measures and continuous functions. 
Suppose $F$ and $G$ are in the domain of $\Lambda _q$,
\[ F = \sum_{j=1}^m \alpha _j 1_{\Omega _F(j)}, \quad G = \sum_{k=1}^n \beta _k 1_{\Omega _G(k)}.\]
Then by \thmref{symmetry}
\[\langle \Lambda _qF,G\rangle  = \sum_{j} \alpha _j \int G \ d \ \Lambda _q 1_{\Omega _F(j)}
= \sum_{j,k} \alpha _j \beta _k \Lambda _q(1_{\Omega _F(j)}, \Omega _G(k))\] 
\[= \sum_{j,k} \alpha _j \beta _k \Lambda _q(1_{\Omega _G(k)}, \Omega _F(j))
= \langle \Lambda _qG,F \rangle ,\] 
establishing the symmetry of $\Lambda _q$.  

A limiting argument as used in \thmref{symmetry} will establish the positivity.
Work in the finite dimensional subspace of functions $F$ which are linear combinations of 
the characteristic functions of a fixed finite collection of clopen sets.
Following \eqref{finapp} and \eqref{fincase}, for each pair $\Omega (i),\Omega (j)$ 
of these clopen sets
\[\Lambda _q(1_{\Omega (i)},\Omega (j)) = \int_{\graph} f_{\Omega (1)} L \psi 
= \lim_n \int_{\graph _n} f_n L \psi \] 
\[ = \sum_{v \in \Omega _n(j)} \partial _{\nu} f_{n,\Omega (i)}(v)
= \sum_{\partial \graph _n } f_{n,\Omega (j)} \partial _{\nu} f_{n,\Omega (i)} .\]
As discussed in \propref{fingcase}, the corresponding finite graph quadratic forms  
$\langle \Lambda _q f_{n},f_n \rangle $ are nonnegative, making the limit nonnegative.

\end{proof}

\bibliographystyle{amsalpha}

\end{document}